\documentclass[12pt]{amsart}
\usepackage[centertags]{amsmath}
\usepackage{amsfonts,amssymb}
\usepackage{graphicx}
\usepackage{enumerate}

  \newtheorem{theorem}{Theorem}
  \newtheorem{corollary}{Corollary}
  \newtheorem{proposition}{Proposition}
  \newtheorem{lemma}{Lemma}%
  \theoremstyle{remark}
  \newtheorem{remark}{Remark}
\begin{document}

\title{On primes and practical numbers}

\date{\today}

\author{Carl Pomerance}
\address{Mathematics Department, Dartmouth College, Hanover, NH 03784}
\email{carl.pomerance@dartmouth.edu}
\author{Andreas Weingartner} 
\address{Department of Mathematics, Southern Utah University, Cedar City,
UT 84720}
\email{weingartner@suu.edu}
\begin{abstract}
A number $n$ is {\it practical} if every integer in $[1,n]$ can be expressed
as a subset sum of the positive divisors of $n$.  We consider the distribution
of practical numbers that are also shifted primes, improving a theorem of
Guo and Weingartner.  In addition, essentially
proving a conjecture of Margenstern, we show that all large odd numbers
are the sum of a prime and a practical number.  We also consider an analogue
of the prime $k$-tuples conjecture for practical numbers, proving the ``correct"
upper bound, and for pairs, improving on a lower bound of Melfi.
\end{abstract} 
\maketitle
\vskip-30pt
\newenvironment{dedication}
        {\vspace{6ex}\begin{quotation}\begin{center}\begin{em}}
        {\par\end{em}\end{center}\end{quotation}}
\begin{dedication}
{In memory of Ron Graham (1935--2020)\\ and Richard Guy (1916--2020)}
\end{dedication}
\vskip20pt

\section{Introduction}

After Srinivasan \cite{S}, we say a positive integer $n$ is {\it practical}
if every integer $m\in[1,n]$ is a subset-sum of the positive divisors of $n$.
After the proof of Erd\H os \cite{EP} in 1950 that the practical numbers have
asymptotic density 0, their 
distribution has been of some interest, with
work of Margenstern, Melfi, Tenenbaum, Saias, and the second-named author of this paper.
In particular, we now know, \cite{W1},  \cite{W2}, that there is a constant $c=1.33607\dots$
such that the number of practical numbers in $[1,x]$ is $\sim cx/\log x$
as $x\to\infty$.  For other problems and results about practical numbers
see \cite[Sec.\ B2]{Guy}.

The problem of how frequently a shifted prime $p-h$ can be
practical was considered recently in \cite{GW}.  Since practical numbers
larger than 1 are all even, one assumes that the shift $h$ is a fixed odd integer.
Under this assumption, it would make sense that the concept of being
practical and being a shifted prime are ``independent events" and 
so it is natural to conjecture that the number
of primes $p\le x$ with $p-h$ practical is of magnitude $x/\log^2 x$.
Towards this conjecture it was shown in \cite{GW} that the number
of shifted primes up to $x$ that are practical is, for large $x$ depending
on $h$, between
$$ 
\frac{x}{(\log x)^{5.7683...}} \hbox{ and } \frac{x}{(\log x)^{1.0860...}}.
$$
Here we make further progress with this problem, proving the conjecture
for the upper bound of the count and reducing the lower bound exponent
$5.7683\dots$  to $3.1647\dots$ .

As in \cite{GW} we consider a somewhat more general problem. 
Let $\theta$ be an arithmetic function with $\theta(n)\ge2$ for all $n$
and let $\mathcal{B}_\theta$ be the set of positive integers containing $n=1$ and all those $n \ge 2$ with  canonical prime
factorization $n=p_1^{\alpha_1}p_2^{\alpha_2}\cdots p_k^{\alpha_k}$, $p_1<\ldots<p_k$, $\alpha_1,\dots,\alpha_k\ge1$,
which satisfy 
\begin{equation}
\label{Bdef}
p_j \le \theta(p_1^{\alpha_1}\dots p_{j-1}^{\alpha_{j-1}})
\qquad (1\le j \le k).
\end{equation}
(It is not necessary that $p_i$ be the $i$-th prime number.)  
Stewart \cite{Stew} and Sierpinski \cite{Sier} showed that
if $\theta(n)=\sigma(n)+1$, where $\sigma(n)$ is the sum of the positive divisors of $n$, then the set $\mathcal{B}_\theta$ is
precisely the set of practical numbers.
Tenenbaum \cite{Ten86} found that if $\theta(n)=yn$, where $y\ge2$ is a
constant, then $\mathcal{B}_\theta$ is the set of integers with $y$-dense divisors; i.e., the ratios of consecutive divisors are at most $y$.  

Throughout this paper, all constants implied by the big $O$ and $\ll$ notation may depend on the choice of $\theta$.   
 For several of our results we assume that there are constants $A,C$ such that
 \begin{equation}\label{thetacond}
   \theta(mn) \le C m^A \theta(n), \qquad m, n \ge 1.
\end{equation}
This holds for $\theta(n)=\sigma(n)+1$ with $A=2$, $C=1$, since we
trivially have $\sigma(mn)\le\sigma(m)\sigma(n)$ and $\sigma(m)\le m^2$.

We write $\log_2 x = \log \log x$ for $x>e^e$ and $\log_2x=1$ for
$0<x\le e^e$, and write $\log_3 x = \log_2 \log x$ for $x>1$.   
Let 
$$l(x)=\exp \left( \frac{\log x}{\log_2 x \log_3^3 x}\right)$$
and
$$ 
S_h(x):= |\{p \le x: p \mbox{ prime},\  p-h \in \mathcal{B}_\theta \}|.
$$
\begin{theorem}\label{thm}
Fix a nonzero integer $h$.
Assume \eqref{thetacond} and $n\le \theta(n) \ll n l(n)$ for $n\ge 1$.
For $x$ sufficiently large depending on the choice of $\theta,h$, we have 
\begin{equation}\label{thm1bounds}
\frac{x}{(\log x)^{3.1648}} <  S_h(x) \ll_h \frac{x}{(\log x)^2},
\end{equation}
where $h\in \mathbb{Z}$ and $h$ is not divisible by $\prod_{p\le \theta(1)} p$ in the lower bound.  
\end{theorem}

The exponent in the lower bound can be taken as any number larger than
$(e+1)\log(e+1)-e+1$.
In the case of practical numbers, where $\theta(n)=\sigma(n)+1$ and $\prod_{p\le \theta(1)}p = 2$, Theorem \ref{thm} implies the following.
\begin{corollary}
For any fixed odd $h\in \mathbb{Z}$, the number of primes $p\le x$ such that $p-h$ is practical satisfies \eqref{thm1bounds}.
\end{corollary}

It seems likely that the upper bound in \eqref{thm1bounds} is best possible, apart from optimizing the implied constant as a function of the shift parameter~$h$.  Our proof shows that this constant is $\ll h/\varphi(h)$. 

Margenstern \cite[Conjecture 7]{Mar} conjectured that every natural number other than $1$ is the sum of two numbers
that are either practical or prime. 
The case of even numbers was settled by Melfi \cite[Theorem 1]{Mel}, 
who showed that every even number is the sum of two practical numbers. 
Somewhat weaker versions of the problem for odd numbers
were recently stated by Sun \cite{Sun}.
(Also see \cite{Sun2} for several other related problems.)
We show that, in the case of odd numbers, there are at most a finite number of exceptions to Margenstern's conjecture.   Tom\'as Oliveira e Silva has told
us that Margenstern's conjecture has no counterexamples to $10^9$ and we
have verified this via a direct search.  We have used this result to bootstrap
the calculation to a considerably higher bound, see Section \ref{sec5}.
It may be difficult by our methods to get a numerical
bound $x_0$ for which every odd number $>x_0$ is the sum of a prime
and a practical number, but such a calculation is  tractable using
our proof if one is prepared to use the extended Riemann Hypothesis
in place of the Bombieri--Vinogradov theorem.  However, it may be that even
this hypothetical $x_0$ is too large for a feasible calculation to close the gap.
\begin{theorem}\label{thm2}
Assume $\theta(n)\ge n$. 
Every sufficiently large integer not divisible by $\prod_{p\le \theta(1)} p$ is the sum of a prime
and a member of $\mathcal{B}_\theta$. 
\end{theorem}
\begin{corollary}\label{cor2}
Every sufficiently large odd integer is the sum of a prime and a practical number. 
\end{corollary}

Margenstern \cite[Theorem 6]{Mar} showed that for every fixed even number $h$, there are infinitely many
practical numbers $n$ such that $n+h$ is also practical. 
He conjectured \cite[Conjecture 2]{Mar} that the number of practical pairs $\{n,n+2\}$ up to $x$
 is asymptotic to $cx/\log^2 x$ for some positive constant $c$. 
Let
$$
T_h(x):=|\{n\le x: n \in \mathcal{B}_\theta, n+h \in \mathcal{B}_\theta\}|.
$$
\begin{theorem}\label{thm3}
Fix a nonzero integer $h$. 
Assume \eqref{thetacond} and $ \theta(n) \ll n l(n)$ for $n\ge 1$.
\newline\noindent (i)
We have
\begin{equation}\label{thm3ub}
 T_h(x) \ll_h \frac{x}{\log^2 x}.
\end{equation}
\newline\noindent (ii)
Assume further that $\theta(n) \ge n$ for all $n$, and that
$n\in \mathcal{B}_\theta$ and $m\le 3n/|h|$ imply $mn\in \mathcal{B}_\theta$. 
Moreover, if $\theta(1)<3$, assume that
\begin{equation}\label{hcond}
  \begin{cases}
h \in 2\mathbb{Z} & \mbox{if} \  \theta(2)\ge 3,  \\
h \in 4\mathbb{Z} & \mbox{if} \ \theta(2)< 3.
\end{cases} 
\end{equation}
Then for  sufficiently large $x$, depending
on the choice of $h$, 
\begin{equation}\label{thm3lb}
T_h(x)>\frac{x}{(\log x)^{9.5367}}.
\end{equation}
\end{theorem}

When $h\in 2\mathbb{Z}$ and  $\theta(n)=\sigma(n)+1$, all conditions of Theorem \ref{thm3} are satisfied,
since for practical $n$ we have $\sigma(n)+1 \ge 2n$, by \cite[Lemma 2]{Mar}.
\begin{corollary}\label{cor3}
For every nonzero even integer $h$, the number of practical $n$ up to $x$,
such that $n+h$ is also practical, satisfies \eqref{thm3ub} and \eqref{thm3lb}.
\end{corollary}

Corollary \ref{cor3} improves on the lower bound by Melfi \cite[Thm. 1.1]{Mel2} for twin practical numbers,
$T_2(x) \gg x/\exp(k\sqrt{\log x})$ for $k>2+\log(3/2)$.

The upper bound in Theorem \ref{thm3} generalizes as follows
 to the distribution of practical $k$-tuples. 

\begin{theorem}\label{thm4}
\label{thm:ktuples}
Fix integers $0\le h_1<h_2<\ldots <h_k$.
Assume \eqref{thetacond} and $ \theta(n) \ll n l(n)$ for $n\ge 1$.
We have
$$\bigl|\{ n\le x: \{n+h_1,\ldots, n+h_k\}\subset \mathcal{B}_\theta \}\bigr| \ll_{h_1,\ldots,h_k}  \frac{x }{\log^k x}.$$
\end{theorem}

When $k\ge3$ getting a lower bound of the same quality for these 
$k$-tuples seems difficult.  In some cases with the practical numbers we
know there are no large examples, such as when the $h_i$ do not all have
the same parity, or for the example $0,2,4,6$ when at least one of $n+h_i$
must be 2 (mod 4) and not divisible by 3, cf.\ \cite{Mar}.  However, when the $k$-tuple
is admissible, i.e., not ruled out by congruence conditions, it would seem
likely that the ``independent events" heuristic would again apply and that
the upper bound in Theorem \ref{thm4} is correct up to a constant factor.
In our proof of the lower bound in Theorem \ref{thm3} we use the 
Bombieri--Vinogradov theorem.  If instead the Elliott--Halberstam conjecture
is assumed, it may be possible to get a reasonable lower bound in
Theorem \ref{thm4} when the $k$-tuple is admissible in the sense above.
Finally, we remark that in certain special cases, such as when the $h_i$
are $0,2,4$, we at least know that there are infinitely many 
practical examples, see Melfi \cite{Mel}.

\section{The upper bound of Theorem \ref{thm}}

\begin{lemma}\label{lemub}
There exists a constant $K>0$ such that for all $a, b\in \mathbb{Z}\setminus \{0\}$ and all $x>1$ we have
$$|\{m\le x: m \text{ and } am+b \text{ are both prime}\}|\le K \frac{a|b|}{\varphi(a|b|)} \cdot \frac{x}{\log^2 x}.$$ 
\end{lemma}
\noindent This result follows immediately from \cite[Lemma 5]{RV}.

Let $P^+(n)$ denote the largest prime factor of $n>1$ and $P^+(1)=1$.  Define
$$ B(x,y,z) = | \{ n\le x: n \in \mathcal{B}_{z\theta}, P^+(n)\le y \}|.$$
\begin{proposition}\label{prop1}
Assume $\theta(n) \ll n\, l(n)$. For $x \ge 2$,  $y\ge 2$ and $z\ge 1$, 
$$B(x,y,z) \ll \frac{x \log(2z)}{\log x} e^{-u/3},$$
where $u=\log x / \log y$. 
\end{proposition}

Before proving this we establish some consequences.
\begin{corollary}\label{cor1}
Let $\alpha \in \mathbb{R}$. Assume \eqref{thetacond} and $ \theta(n) \ll n\, l(n)$ for $n\ge 1$.
  For $x\ge 1$, $y\ge 2$, $z\ge 1$, 
$$ \sum_{n\le x,\ n \in \mathcal{B}_{z\theta} \atop P^+(n) \le y} \left(\frac{\sigma(n)}{n} \right)^\alpha
\ll_\alpha  \frac{x \log(2z)}{\log(2x)} \exp\left(-\frac{\log x}{3\log y}\right) .$$
\end{corollary}
\begin{proof}
When $\alpha\le 0$, the result follows from Proposition \ref{prop1}. 
We will show the result for $\alpha \in \mathbb{N}$ by induction. 
Note that because of \eqref{thetacond} we have that $kd  \in \mathcal{B}_\theta$
implies $k \in \mathcal{B}_{\theta_d}$, where $\theta_d(n) = C d^A \theta(n) $.
By Proposition \ref{prop1} with $z$ replaced by $z Cd^A$,
\begin{equation*}
\begin{split}
 \sum_{n\le x, \,  n \in \mathcal{B}_{z\theta} \atop P^+(n)\le y} & \left(\frac{\sigma(n)}{n} \right)^\alpha
  =  \sum_{n\le x, \,  n \in \mathcal{B}_{z\theta} \atop P^+(n)\le y }  
\left(\frac{\sigma(n)}{n} \right)^{\alpha-1}\sum_{d|n} \frac{1}{d} \\
&  \le \sum_{d\le x} \frac{\sigma(d)^{\alpha-1}}{d^\alpha} \sum_{k\le x/d, \,  k \in \mathcal{B}_{z\theta_d} \atop P^+(k)\le y} \left(\frac{\sigma(k)}{k} \right)^{\alpha-1} \\
&  \ll_\alpha \sum_{d\le x} \frac{\sigma(d)^{\alpha-1}}{d^\alpha} 
\frac{x \log (2dz)}{d\log (2x/d)} \exp\left(-\frac{\log (x/d)}{3\log y}\right) \\
 &    \ll  x  \exp\left(-\frac{\log x}{3\log y}\right) \sum_{d\le x }\exp\left(\frac{\log d}{3\log y}\right) \frac{(\log_2 d)^{\alpha-1}\log(2dz)}{d^{2} \log (2x/d)}\\
&   \ll_\alpha \frac{x \log(2z)}{\log(2x)}  \exp\left(-\frac{\log x}{3\log y}\right),
\end{split}
 \end{equation*} 
since $\exp((\log d)/(3\log y))\le d^{1/2}$.
\end{proof}

With $y=x$, $z=1$ and $\alpha=1$ in Corollary \ref{cor1}, we get
\begin{corollary}\label{corsigma}
Under the assumptions of Corollary \ref{cor1} we have, for $x> 1$, 
$$ \sum_{n\le x \atop n \in \mathcal{B}_\theta} \frac{\sigma(n)}{n} \ll  \frac{x}{\log x}.$$
\end{corollary}

\begin{remark}
Corollary \ref{corsigma} allows us to replace the relative error term $O( \log_2 x /\log x)$ in \cite[Theorem 1.1]{W1},
 the asymptotic for the count of practical numbers up to $x$, by $O(1/\log x)$. 
 Indeed, in the proof of \cite[Theorem 1.1]{W1}, the estimate $\sigma(n)/n \ll \log_2 n$ leads to the extra
 factor of $\log_2 x$. Using instead Corollary \ref{corsigma} in the proofs of Lemmas 5.3 and 5.6 of \cite{W1},
 the factor $\log_2 x$ can be avoided.
\end{remark}

\begin{proof}[Proof of the upper bound in Theorem \ref{thm}]

Assume $x\ge 2|h|$.
We consider those $n\in\mathcal{B}_\theta$ with $n+h$ prime and $n+h\le x$.
We may assume that $n>x/\log^2x$.
Write $n=mq$, where $q=P^+(n)$. We have $m\in \mathcal{B}_\theta$, $P^+(m)\le q$ and $q \le \theta(m)\le ml(m)$. 
So, assuming $x$ is large, we have $m>x^{1/3}$.  By Lemma \ref{lemub},
\begin{equation*}
\begin{split}
 S_h(x) & \le \sum_{m\in \mathcal{B}_\theta} |\{q ~ {\rm prime}: mq+h\, {\rm prime},\,  q\le (x-h)/m\}| \\
 & \ll  \sum_{m\in \mathcal{B}_\theta, \ m> x^{1/3} \atop m P^+(m) \le x-h} 
 \frac{m|h|}{\varphi(m|h|)}  \frac{(x-h)/m}{\log^2(2(x-h)/m)} \\
 & \le   \frac{2|h| x}{\varphi(|h|)} \sum_{m\in \mathcal{B}_\theta, \  m> x^{1/3} } 
 \frac{1}{\varphi(m)\log^2 P^+(m)}.
\end{split}
 \end{equation*} 
 We will show that the last sum is $\ll 1/\log^2 x$.
  With $p=P^+(m)$ and $m=kp$, we have $k\in \mathcal{B}_\theta$ and 
 $k>x^{1/7}$. The last sum is 
 $$
 \ll \sum_{p\ge 2} \frac{1}{p\log^2 p} \sum_{k \in \mathcal{B}_\theta, \ k>x^{1/7} \atop P^+(k)\le p} 
 \frac{k}{\varphi(k)}\cdot\frac{1}{k}.
  $$ 
Since $k/\varphi(k) \ll \sigma(k)/k$, Corollary \ref{cor1} (with $\alpha=z=1$) and partial summation applied to the inner sum shows that
the last expression is 
$$
 \ll \sum_{p\ge 2} \frac{1}{p\log^2 p}\cdot \frac{\log p}{\log x} \exp\left(-\frac{\log x}{21\log p}\right)
\ll \frac{1}{\log^2 x},
 $$ 
 by the prime number theorem.
\end{proof}

\begin{proof}[Proof of Proposition \ref{prop1}]
We follow the proof of Saias \cite[Prop. 1]{EDD1}, who established this result in the case when $\theta(n)=yn$ with $y\ge2$ (integers with
$y$-dense divisors) and in the case when $\theta(n)=\sigma(n)+1$ (practical numbers) and $z=1$. 
Let $f(n)$ be an increasing function with $\theta(n) \le n f(n)$ for all $n\ge 1$ and $f(n)\ll l(n)$.  Suppose $n\in\mathcal{B}_{z\theta}$, where $n=p_1p_2\dots p_k$
with $p_1\le p_2\le\dots\le p_k$.  Since $f$ is increasing,
$p_j \le z p_1 \cdots p_{j-1} f (p_1 \cdots p_{j-1})$, so $p_j^2 \le z n f(n) \le z x f(x)$ for $n\le x$. By sorting the integers counted in $B(x,y,z)$ according to their largest prime factor, we get
$$ B(x,y,z) \le 1 + \sum_{p \le \min(y, \sqrt{z x f(x)})} B(x/p,p,z),$$
the analogue of \cite[Lemma 8]{EDD1}. 

Let $\Psi(x,y)$ denote the number of integers $n\le x$ with $P^+(n)\le y$.
We write $u=\log x / \log y$ and $\tilde{v}=\log x / \log(2z)$. Let
$\tilde{\rho}(u)=\rho(\max\{0,u\})$, where $\rho(u)$ is Dickman's function. 
Let $\tilde{D}(x,y,z)$ be the function defined in \cite[p. 169]{EDD1}.
It satisfies
$$ \tilde{D}(x,y,2z) \asymp \frac{x}{\tilde{v}} \tilde{\rho}\Bigl(u\bigl(1-1/\sqrt{\log y}\bigr)-1\Bigr) \qquad (0<u<3(\log x)^{1/3})$$
and 
$$   \tilde{D}(x,y,2z) = \Psi(x,y) \qquad (u\ge 3(\log x)^{1/3}),$$
Lemma 9 of \cite{EDD1} shows that 
$$ \tilde{D}(x,y,2z) \ge 1 +  \sum_{p \le \min(y, \sqrt{2z x} l(x))} \tilde{D}(x/p,p,2z),$$
for $z\ge 1$, $y\ge 2$, $\tilde{v}\ge v_0$ and $0<u\le 3 (\log x)^{1/3}$.

We claim that 
\begin{equation}\label{BD}
 B(x,y,z) \le c \tilde{D}(x,y,2z),
 \end{equation}
for some suitable constant $c$. 
If $2\le x \le x_0$, we have $\tilde{D}(x,y,2z) \asymp 1$, so we may assume $x\ge x_0$ and hence $\sqrt{f(x)}\le l(x)$. 
If $0<\tilde{v}\le u < 3(\log x)^{1/3}$, then $2z\ge y$ and $B(x,y,z)=\Psi(x,y)\ll \tilde{D}(x,y,2z)$, where the last estimate is 
derived in the penultimate display on page 182 of \cite{EDD1}. 
If $0<u\le \tilde{v}\le v_0$, then $\tilde{D}(x,y,2z) \asymp x$ so \eqref{BD} holds. 
If $ u \ge  3(\log x)^{1/3}$, then $\tilde{D}(x,y,2z) = \Psi(x,y)$ and \eqref{BD} holds. 
Assume that $c$ is such that \eqref{BD} holds in the domain covered so far. 
In the remainder we may assume that $u\le 3 (\log x)^{1/3}$ and $\tilde{v}\ge v_0$. 
We show by induction on $k$ that \eqref{BD} holds for
$y\ge 2, z\ge 1, 2\le x\le 2^k$. 
We have
 \begin{equation*}
\begin{split}
 B(x,y,z)  & \le 1 + \sum_{p \le \min(y, \sqrt{z x f(x)})} B(x/p,p,z) \\
 & \le 1 + c  \sum_{p \le \min(y, \sqrt{z x f(x)})} \tilde{D}(x/p,p,2z) \\
 & \le c\left(1 +   \sum_{p \le \min(y, \sqrt{2z x}l(x) )} \tilde{D}(x/p,p,2z) \right) \\
 & \le c  \tilde{D}(x,y,2z).
\end{split}
 \end{equation*} 
 It remains to show that 
 $$\tilde{D}(x,y,2z) \ll x\frac{\log(2z)}{\log x} e^{-u/3}.$$
 We may assume $x\ge x_0$. If $u\le 3(\log x)^{1/3}$, then $y \ge y_0$ and the result follows from $\rho(u)\ll e^{-u}$. 
 If $u> 3(\log x)^{1/3}$, then 
 $$\tilde{D}(x,y,2z) = \Psi(x,y) \ll x e^{-u/2} \ll \frac{x}{\log x} e^{-u/3} \ll \frac{x \log(2z)}{\log x} e^{-u/3},$$
 where the upper bound for $\Psi(x,y)$ is \cite[Thm.\ III.5.1]{Ten}.
\end{proof}

\section{Some Lemmas}

The following observation follows immediately from the definition of the set $\mathcal{B}_\theta$ in \eqref{Bdef}.
\begin{lemma}\label{Lem-cor}
Let $\theta(n)\ge n$ for all $n\in \mathbb{N}$. If $n\in \mathcal{B}_\theta$ and $P^+(k) \le n$, then $nk \in \mathcal{B}_\theta$.
\end{lemma}

If $\theta(n)=yn$, we write $\mathcal{D}_y$ for $\mathcal{B}_\theta$. 
For an integer $n>1$, let $P^-(n)$ denote the least prime dividing $n$,
and let $P^-(1)=+\infty$.

\begin{lemma}\label{Saias}
There is a number $y_0$ such that if
$x\ge z^4\ge 1$ and  $y \ge \max\{y_0,z+z^{0.535}\}$, 
we have
$$
|\{ n \le x: n \in \mathcal{D}_y,\, P^-(n)>z\}| \asymp \frac{x \log (y/z)}{\log (xy) \log (2z)}.
$$
This conclusion continues to hold if
$z+1\le y\le y_0$ and $(z,y]$ contains at least one prime number.
\end{lemma}
\begin{proof}
When $x\ge y\ge y_0$ and $z\ge 3/2$, then $\log(xy) \asymp \log x$ and the result follows from \cite[Thm. 1]{EDD2}
and \cite[Rem. 2]{IDD1}.
When $y>x$, the result follows from $|\{ n \le x: P^-(n)>z\}| \asymp x/\log(2z)$. 
If  $1\le z \le 3/2$, the result follows from \cite[Thm. 1]{EDD1}.
If $y\le y_0$, the result follows from iterating \cite[Lemma 8]{EDD2} a finite number of times. 
\end{proof}


\begin{lemma}\label{lemdmult}
For $d\in \mathbb{N}$, $x\ge 1$, $z \ge 1$ and  $y \ge 2z$, we have
$$
|\{ n\le x: n \in \mathcal{D}_y,\, P^-(n)>z,\, d\mid n \}| 
\ll 1_{d\in \mathcal{D}_y} +\frac{x \log( dy)}{d \log( xy )\log (2z)}.
$$
\end{lemma}

\begin{proof}
We first assume that $x/d\ge z^4$. 
If $d=1$ the result follows from Lemma \ref{Saias}, so we assume $d>1$.
We have
\begin{equation*}
\begin{split}
|\{dw\le x: dw \in \mathcal{D}_y, P^-(w)>z \}| & \le
|\{w \le x/d: w \in \mathcal{D}_{d y},  P^-(w)>z \}| \\
& \ll \frac{x \log (dy)}{d \log (xy) \log (2z)},
\end{split}
\end{equation*}
by Lemma \ref{Saias}.

If $x/d\le z^4$, then $\log(xy) \le \log(y d z^4 )\le 5 \log(yd )$, so the result follows from 
$|\{2\le w\le x/d: P^-(w)>z \}|  \ll x/(d \log(2z)).$
\end{proof}

\begin{lemma}\label{lemrelp}
Assume $\theta(n)\ge n$ for all $n \in \mathbb{N}$. 
For all $h\in \mathbb{N}$ that are not divisible by $\prod_{p\le \theta(1)}p$, we have
$$
|\{x/p_0< n\le x: n \in \mathcal{B}_\theta,\, \gcd(n,h)=1\}| \gg \frac{x}{ \log x \log (2h) \log_2 h},
$$
for $x\ge K  \log^5( 2h)$, where $p_0\le \theta(1)$ is the smallest prime not dividing $h$,
and $K$ is some positive constant depending only on $\theta$.
Moreover, there exists a constant $\eta >0$ such that if $L\ge 1$ satisfies
$$ \sum_{p|h, \ p>L} \frac{\log p}{p} < \eta$$
then, for $x \ge K L^5$,
$$
|\{x/p_0< n\le x: n \in \mathcal{B}_\theta,\, \gcd(n,h)=1\}| \gg \frac{x}{L \log x  \log(2L)}.
$$
\end{lemma}

\begin{proof}
Let $p_0\le \theta(1)$ be the smallest prime with $p_0 \nmid h$.
Let $k\in \mathbb{N}$, $L_k=p_0^k/2$, and assume $x\ge2L_k^5$. Since $\theta(n)\ge n$,
\begin{equation*}
\begin{split}
 |\{  x/p_0<n=p_0^k w \le & x: n \in \mathcal{B}_\theta,\, P^-(w)>L_k\}| \\ 
&\ge  |\{ x/p_0^{k+1}< w\le x/p_0^k: w \in \mathcal{D}_{p_0^k},\,P^-(w)>L_k\}|.
\end{split}
\end{equation*}
We would like to use Lemma \ref{Saias} to obtain a lower bound for this count, but
the fact that $w$ is not free to roam over the entire interval $[1,x/p_0^k]$ is
problematic.  We note though that Lemma \ref{Saias} implies there is a set
${\mathcal K}\subset{\mathbb N}$ with bounded gaps such that if $x\ge 2L_k^5$
and $k\in{\mathcal K}$, we have
\begin{align*}
 |\{ x/p_0^{k+1}< w\le x/p_0^k: w \in \mathcal{D}_{p_0^k},\,P^-(w)>L_k\}|
&\gg  \frac{x \log(p_0^k/L_k)}{p_0^k\log x \log L_k} \\
&\asymp \frac{x}{L_k \log x \log L_k}.
\end{align*}

We have
\begin{align*}
 & |\{ w\le x/p_0^k: w \in \mathcal{D}_{p_0^k}, \,P^-(w)>L_k, \,\gcd(h,w)>1\}|\\
&\hskip2cm\le  \sum_{p|h \atop p>L_k} |\{ w\le x/p_0^k: w \in \mathcal{D}_{p_0^k},  \,P^-(w)>L_k,\, p\mid w \}|\\
&\hskip2cm\ll   \sum_{p|h \atop L_k<p \le 2 L_k} 1 + \sum_{p|h \atop p>L_k} \frac{x \log p}{L_k p \log x \log L_k },
\end{align*}
by Lemma \ref{lemdmult}, since $\log(pp_0^k)\ll\log p$ for $p>L_k$.
The sum of $1$ is clearly $\le L_k \le (x/2)^{1/5}$. 
The second statement of the lemma now follows with the smallest $k \in \mathcal{K}$ 
such that $L_k\ge L$. 

Since $h$ has at most $\log h / \log L_k$ prime factors $>L_k$, the last sum above is
$$
\ll \frac{\log h}{\log L_k} \cdot \frac{x}{L_k\log L_k  \log x} \cdot  \frac{\log L_k}{L_k} = \frac{x \log h}{L_k^2 \log L_k \log x}.
$$
We need this to be $< x/(C L_k \log x \log L_k)$ for some sufficiently large constant $C>0$,
that is, $L_k\ge C \log(2h)$. 
The first statement of the lemma now follows with the smallest such $k \in \mathcal{K}$.
\end{proof}

\section{The lower bound of Theorem \ref{thm}}\label{sec2}

Let $h$ be a fixed integer that is not a multiple of $\prod_{p\le \theta(1)}p$. 
Let $\delta=1/\log_2x$ and define
\[
\mathcal{Q}=\{q\in(x^{1/2-\delta},x^{1/2}/\log^{10}x]:\ \gcd(q,h)=1,\ q\in\mathcal{B}_\theta\}.
\]
Let ${\mathcal N}_h(x)$ denote the set of pairs $(q,m)$ with $q\in{\mathcal Q}$,
$qm+h\le x$, and $qm+h$ prime, and let $N_h(x)=|{\mathcal N}_h(x)|$.  Thus,
\[
N_h(x)=\sum_{q\in{\mathcal Q}}\pi(x;q,h).
\]
Now, by the Bombieri--Vinogradov theorem, see \cite[p. 403]{Ten}, we have
\begin{equation*}
\sum_{q\in{\mathcal Q}}\Bigg|\pi(x;q,h)-\frac{\pi(x)}{\varphi(q)}\Bigg|
\ll\frac x{\log^6x}.
\end{equation*}
 Thus, 
 \[
N_h(x)=\sum_{q\in{\mathcal Q}}\pi(x;q,h)=\sum_{q\in{\mathcal Q}}\frac{\pi(x)}{\varphi(q)}
+O\Big(\frac x{\log^6x}\Big).
\]
Further, using Lemma \ref{lemrelp}, we have
\[
\sum_{q\in{\mathcal Q}}\frac1{\varphi(q)}\ge\sum_{q\in{\mathcal Q}}\frac1q\gg_h\delta.
\]
We conclude that 
\begin{equation}
\label{eq:Nh}
N_h(x)\gg_h\delta x/\log x.
\end{equation}

Let ${\mathcal N}_{h,1}(x)$ denote the set of those pairs $(q,m)$ in ${\mathcal N}_h(x)$
with $x^\delta<P^+(m)<x^{1/2-\delta}$.
\begin{lemma}
\label{lem:P(m)}
We have $|{\mathcal N}_{h,1}(x)|=|{\mathcal N}_h(x)|+O(\delta^2x/\log x)$,
\end{lemma}
\begin{proof}
Let $q\in{\mathcal Q}$.  The number of integers $m\le (x-h)/q$ with
$P^+(m)\le x^\delta$ is $\ll (x-h)/(q\log^{10}x)$, see \cite[Lem.\ III.5.19]{Ten},
and so such numbers $m$ are negligible.  For $m=rk$, where
$r=P^+(m)\ge x^{1/2-\delta}$, we have $k\le x^{2\delta}$.  Thus, the number
of such pairs $(q,rk)$ is at most
\[
\sum_{q\in{\mathcal Q}}\sum_{k\le x^{2\delta}}\sum_{\substack{r\le (x-h)/qk\\r\,{\rm prime}\\qrk+h\,{\rm prime}}}1.
\]
The inner sum, by Lemma \ref{lemub}, is $\ll_hx/(\varphi(q)\varphi(k)\log^2x)$.
Summing on $k$ gives us $\ll_h\delta x/(\varphi(q)\log x)$, and then
summing on $q$ gives us $\ll_h \delta^2x/\log x$, using $q/\varphi(q)\ll\sigma(q)/q$,
Corollary \ref{corsigma}, and partial summation.  This concludes the proof.
\end{proof}

\begin{corollary}
\label{cor:inB}
For a pair $(q,m)$ in ${\mathcal N}_{h,1}(x)$ we have $qm\in{\mathcal B}_\theta$.
\end{corollary}
\begin{proof}
Since $P^+(m)<x^{1/2-\delta}<q$, it follows
from Lemma \ref{Lem-cor} that $qm\in{\mathcal B}_\theta$.  
\end{proof}

Let $v_2(n)$ denote the number of factors 2 in the prime factorization of $n$
and let $\Omega(n)$ denote the total number of prime factors of $n$, counted
with multiplicity.
Let $\varepsilon>0$ be arbitrarily small but fixed.  Let
${\mathcal N}_{h,2}(x)$ denote the set of pairs $(q,m)\in{\mathcal N}_{h,1}(x)$
with 
$$\Omega(m)\le I:=\lfloor(1+\varepsilon)\log_2x\rfloor \hbox{ and } v_2(m)\le4\log_3x.$$
 \begin{lemma}\label{EE}
 We have
 \[
| {\mathcal N}_{h,2}(x)|=|{\mathcal N}_h(x)|+O_h(\delta^2x/\log x).
\]
\end{lemma}
\begin{proof}
Assume $(q,m)\in{\mathcal N}_{h,1}(x)$.  Let $r=P^+(m)$, so that $r>x^\delta$,
and write $m=rk$.  If $(q,m)\notin{\mathcal N}_{h,2}(x)$ then either
$\Omega(k)>I-1$ or $v_2(k)>4\log_3x$.  For a given
number $k$, the number of primes $r\le (x-h)/qk$ with $qrk+h$ prime
is, by Lemma \ref{lemub}, $\ll_h x/(\varphi(q)\varphi(k)\log^2(x/qk))$.
Summing this expression over $k$ with $v_2(k)>4\log_3x$ and $q\in{\mathcal Q}$,
it is $\ll_h\delta^2x/\log x$, since $2^{-4\log_3x}<\delta^2$.  We now wish to
consider the case when $\Omega(k)>I-1$.  Following a standard theme
 (see Exercises 04 and 05 in \cite{HT})
we have uniformly for each real number $z$ with $1<z<2$ that
\begin{equation}
\label{eq:exer}
\sum_{n\le x}\frac{z^{\Omega(n)}}{\varphi(n)}\ll\frac1{2-z}(\log x)^z.
\end{equation}
Applying this with $z=1+\varepsilon$, we have
\begin{equation*}
\sum_{\substack{k\le x^{1/2}\\\Omega(k)>I-1}}\frac1{\varphi(k)}
\le z^{-I+1}\sum_{k\le x^{1/2}}\frac{z^{\Omega(k)}}{\varphi(k)}
\ll (\log x)^{1+\varepsilon-(1+\varepsilon)\log(1+\varepsilon)}.
\end{equation*}
This last expression is of the form $(\log x)^{1-\eta}$,
where $\eta>0$ depends on the choice of $\varepsilon$.
Thus, the number of pairs $(q,m)$ in this case is $\ll_h \delta x/(\log x)^{1+\eta}$,
which is negligible.
\end{proof}

Let $\Omega_3(n)=\Omega(n/v_2(n))$ denote the number of odd prime factors of $n$ counted with
multiplicity, and let ${\mathcal N}_{h,3}$ denote the number of pairs
$(q,m)\in{\mathcal N}_{h,2}$ with $\Omega_3(q)\le J:=\lfloor(e+\varepsilon)\log_2x\rfloor$.
\begin{lemma}
\label{lem:oddprimes}
We have $|{\mathcal N}_{h,3}(x)|=|{\mathcal N}_h(x)|+O_h(\delta^2x/\log x)$.
\end{lemma}
\begin{proof}
By the same method that gives \eqref{eq:exer}, we have
\begin{equation}
\label{eq:exer2}
\sum_{n\le x}\frac{z^{\Omega_3(n)}}{\varphi(n)}\ll\frac1{3-z}(\log x)^z,
\end{equation}
uniformly for $1<z<3$.
Assuming that $\varepsilon$ is small enough that $z=e+\varepsilon<3$, we have
\[
\sum_{\substack{q\in{\mathcal Q}\\\Omega_3(q)>J}}\frac1{\varphi(q)}
\le \sum_{\substack{q\le x^{1/2}\\\Omega_3(q)>J}}\frac1{\varphi(q)}
\le z^{-J}\sum_{q\le x^{1/2}}\frac{z^{\Omega_3(q)}}{\varphi(q)}\ll(\log x)^{z-(e+\varepsilon)\log z}.
\]
Since $z-(e+\varepsilon)\log z=-\eta<0$, where $\eta$ depends on the
choice of $\varepsilon$, this calculation shows that those pairs with
$\Omega_3(q)>J$ are negligible.
\end{proof}

Let $K=\lfloor4\log_3x\rfloor+1$.
For a given pair $(q,m)\in{\mathcal N}_{h,3}(x)$, we count the number of pairs
$(q',m')\in{\mathcal N}_{h,3}(x)$ with $q'm'=qm$.    The pair $(q',m')$ is
determined by $(q,m)$ and $m'$, so all we need to do is count the number
of divisors $d$ of $qm$ with $\Omega(d)\le I$ and $v_2(d)<K$.  This count is at most
\[
K\sum_{i\le I}\binom{I+J}{i}\ll K\binom{I+J}{I}.
\]
 Stirling's formula shows that
$$  K\binom{I+J}{I} \ll (\log x)^{\alpha + \eta}\log_3x,$$
where $\alpha = (e+1)\log(e+1)-e\log e=2.16479...$ and $\eta \to 0$ as $\varepsilon \to 0$.
It follows from \eqref{eq:Nh} and Lemma \ref{lem:oddprimes} that
$$  S_h(x) \gg  \frac{\delta x}{\log x} \cdot \frac{1}{(\log x)^{\alpha + \eta}\log_3x} \gg \frac{x}{(\log x)^{1+\alpha +2\eta}}
=\frac{x}{(\log x)^{3.16479... +2\eta}}.$$

\begin{remark}
The proof of the lower bound of Theorem \ref{thm} would be somewhat simpler
if instead of the Bombieri--Vinogradov theorem we had used a very new result
of Maynard \cite{May}.  With the choice of parameters $\delta=0.02$, $\eta=0.001$
in his Corollary 1.2, one has for the set ${\mathcal Q}$ of integers $q\le x^{0.52}$
with a divisor in $(x^{0.041},x^{0.071})$ that
\[
\sum_{\substack{q\in{\mathcal Q}\\\gcd(q,a)=1}}\left|\pi(x;q,a)-\frac{\pi(x)}{\varphi(q)}\right|
\ll_{a,A}\frac x{\log^Ax},
\]
for any fixed integer $a\ne0$ and any positive $A$.  We note that all of the members
of ${\mathcal B}_\theta\cap(x^{0.041},x^{0.52}]$ are in ${\mathcal Q}$.
\end{remark}

\section{Proof of Theorem \ref{thm2}}
\label{sec5}

Let $h$ be an integer in $(x/2,x]$ that is not a multiple of $\prod_{p\le \theta(1)}p$.
Define
$$
\mathcal{D}= \{ q \in \mathcal{B}_\theta \cap (x^{1/2-\delta}, x^{1/2}/\log^{10} x]:\gcd(q,h)=1 \}.
$$
By Lemma \ref{lemrelp},
\begin{equation}\label{Dprime}
|\mathcal{D}| \gg \frac{ x^{1/2}}{ \log^{12} x \log\log x}.
\end{equation}
For each $q \in \mathcal{D}$, if $p\le x/2<h$, where $p$ is a prime that satisfies $p\equiv h \bmod q$, then 
$p=h-qm$ for some $m\in \mathbb{N}.$
Let $M_h(x)$ denote the number of pairs $(p,q)$ with $p$ prime, $p\le x/2$, $p\equiv h \bmod q$ 
and $q\in \mathcal{D}$.
As in Section \ref{sec2}, we have
$$
M_h(x)= \sum_{q\in{\mathcal D}}\pi(x/2;q,h)=\sum_{q\in{\mathcal D}}\frac{\pi(x/2)}{\varphi(q)}
+O\Big(\frac x{\log^6x}\Big).
$$
 From \eqref{Dprime}, we have
$$
F:=\sum_{q\in{\mathcal D}}\frac1{\varphi(q)}\ge\sum_{q\in{\mathcal D}}\frac1q\ge 
\frac{|{\mathcal D}|}{x^{1/2}/\log^{10}x}\gg\frac{1}{\log^2 x \log\log x}.
$$
We conclude that 
\begin{equation}\label{Mh}
M_h(x)\gg F \frac{x}{\log x} \gg \frac{ x}{\log^3 x \log\log x}.
\end{equation}

We claim that most of the pairs $(p,q)$ counted in $M_h(x)$ are such that $qm=h-p \in \mathcal{B}_\theta$. 
Since $q > x^{1/2-\delta}$ and $qm<h\le x$, we have $m\le x^{1/2+\delta}$. 
If $P^+(m)\le x^{1/2-\delta}$, then $P^+(m)<q$ and $mq \in \mathcal{B}_\theta$. 
If $P^+(m)> x^{1/2-\delta}$, write $r=P^+(m)>x^{1/2-\delta}$ and $m=ra$ with $a<x^{2\delta}$.
Given $a$ and $q$, the number of primes $r<x/(aq)$ with $h-aqr$ prime is
\begin{equation}\label{hqa}
\ll  \frac{hx}{\varphi(h)\varphi(q)\varphi(a) \log^2 x},
\end{equation}
by Lemma \ref{lemub}.
We have $h/\varphi(h) \ll \log\log x$ and
$$\sum_{a < x^{2\delta}} \frac{1}{\varphi(a)} \ll \delta \log x.$$
Thus, summing \eqref{hqa} over $q\in \mathcal{D}$ and $a<x^{2\delta}$ amounts to 
$$
\ll F \frac{x \delta \log\log x}{\log x} = o\left(F \frac{x}{\log x}\right),
$$
since $\delta=1/(\log\log x)^2$.
By \eqref{Mh}, the number of pairs $(p,q)$ with $h=p+qm$, $p$ prime and $qm \in \mathcal{B}_\theta$ is 
$$
\gg  F \frac{x}{\log x} \gg \frac{ x}{\log^3 x \log\log x},
$$
which is at least $1$ when $x$ is sufficiently large. 
This completes the proof of Theorem \ref{thm2}.

\subsection{Checking Margenstern's conjecture numerically}

For positive coprime integers $u,v$, let $p(u,v)$ be the least prime
$p\equiv u\pmod v$, and let $M(v)=\max_{\gcd(u,v)=1}p(u,v)$.
For example, $M(8)=17$, since $p(1,8)=17$, $p(3,8)=3$, $p(5,8)=5$,
and $p(7,8)=7$.
\begin{lemma}
\label{lem:alg}
Suppose that $a$ is a positive integer with $M(2^a)<2^{2a+1}$.
Then every odd number $n\in(M(2^a),2^{2a+1})$ is the sum of a prime
and a practical number.
\end{lemma}
\begin{proof}
For each odd $n\in(M(2^a),2^{2a+1})$ let $q=n-p(n,2^a)$.  Note that
$0<q<2^{2a+1}$ and $2^a\mid q$.  Since $2^a$ is practical
and $\sigma(2^a)+1=2^{a+1}>q/2^a$, it follows that $q$ is practical.  Thus,
$n=q+p(n,2^a)$ is a representation of $n$ as the sum of a prime
and a practical.
\end{proof}

Note that the condition in Lemma \ref{lem:alg} that $M(2^a)<2^{2a+1}$
is not guaranteed by any known result in analytic number theory.
We do know that $M(2^a)\le 2^{O(a)}$ with a fairly modest $O$-constant,
but we are not close to proving the condition in the lemma.  
(Heuristically, we should have $M(2^a)=O(2^aa^2)$.)  For a given numerical
value of $a$, one might actually compute the exact value of $M(2^a)$.
And if f it is smaller than $2^{2a+1}$, we have verified Margenstern's
conjecture for the interval $(M(2^a),2^{2a+1})$.  For example, since
$M(2^3)=17$, we automatically have the conjecture for odd numbers
in the interval $(17,128)$.

We have computed that $M(2^{23})=997{,}427{,}777$.  This number is
less than $2^{47}$, in fact, it is less than $10^9$.  Thus, Margenstern's
conjecture holds for all odd numbers (greater than 1) up to $2^{47}$.  
Moreover, since $M(2^{35})=9{,}968{,}601{,}716{,}713 < 2^{47}$,
the conjecture holds up to $2^{71}$.
It would not be difficult to push this calculation further.

\section{The upper bound in Theorems \ref{thm3} and \ref{thm4}}

For a natural number $n$, a divisor $d$ of $n$ is said to be {\it initial} if
$P^+(d)\le P^-(n/d)$.  Let $I_y(n)$ be the largest initial divisor of $n$ with
$d\le y$.  Note that if $n\in{\mathcal B}_\theta$, then $I_y(n)\in{\mathcal B}_\theta$
for all $y$.

Assume $n\le x$ and $n, n+h \in \mathcal{B}_\theta$. 
Let $q=I_{x^{1/3}}(n)$, $q'=I_{x^{1/3}}(n+h)$.  
Since $n,n+h\in{\mathcal B}_\theta$ and $\theta(n)=n^{1+o(1)}$,
we may assume that $q,q'\in [x^{1/7},x^{1/3}]$.
Write $n=qm$ and $n+h=q'm'$. 
We have $q,q' \in \mathcal{B}_\theta$ and $P^-(m)\ge P^+(q)=:r$, $ P^-(m')\ge P^+(q')=:r'$. 
Given $q,q' \in \mathcal{B}_\theta$ with $d=\gcd(q,q')$, we need $m,m'$ such that
$ q'm'-qm=h$. This equation only has solutions if $d|h$, in which case all solutions have the form  
$$ m=m_0+ j q'/d, \quad m'=m'_0+jq/d, \quad j \in \mathbb{Z}.$$
If $m_0,m'_0$ are the smallest positive solutions to $q'm'-qm=h$, then $1\le n=mq \le x$ implies $0\le j \le dx/qq'\le hx/qq'$.
Let 
$$ \mathcal{A}=\{( m_0+ j q'/d)(m'_0+jq/d): 0\le j \le hx/qq'\},$$
and let  $S(\mathcal{A})$ be the number of elements of $\mathcal{A}$ remaining after removing 
all products $m m'$, where either $m$ is a multiple of a  prime $p<r$, $p\nmid hqq'$,
or $m'$  is a multiple of a  prime $p<r'$, $p\nmid hqq'$. 
For each prime $p\nmid hqq'$, each of the conditions $p|m$ and $p|m'$ is equivalent to $j$ belonging to 
a unique residue class modulo $p$ (because $p\nmid qq'$), and those two residue classes are distinct (because $p\nmid h$).
Selberg's sieve \cite[Prop. 7.3 and Thm. 7.14]{ODC} shows that
$$
S(\mathcal{A}) \ll \frac{hx/qq'}{\log r \log r'} \left(\frac{hqq'}{\varphi(hqq')}\right)^2
\ll_h \frac{xqq'}{\varphi(q)^2  \varphi(q')^2 \log P^+(q) \log P^+(q') }.
$$
Summing this estimate over $q,q' \in  [x^{1/7},x^{1/3}] \cap \mathcal{B}_\theta$, the upper bound in Theorem \ref{thm3} follows from Lemma \ref{lem:qsum} with $\alpha=2$. 

This argument generalizes naturally to yield Theorem \ref{thm4}: For $1\le i \le k$, 
let $n+h_i=m_i q_i \in \mathcal{B}_\theta$, where $q_i=I_{x^{1/(k+1)}}(n+h_i)$,
so that $q_i \in \mathcal{B}_\theta \cap [x^{1/(2k+3)},x^{1/(k+1)}]$.
One finds that if $\gcd(q_i,q_l)|(h_l-h_i)$, for $1\le i < l \le k$, then 
$$ m_i = m_{i,0}+j \, \text{lcm}(q_1,\ldots, q_k) /q_i \qquad (1\le i \le k),$$
where $0\le j \le x/ \text{lcm}(q_1,\ldots, q_k)\le \frac{x}{q_1\ldots q_k}\prod_{1\le i < l \le k}(h_l-h_i)$.
Eliminating values of $j$ for which $p|m_i$, where $p<P^+(q_i)$, $p \nmid \prod_{i\le k} q_i$ and $p\nmid \prod_{1\le i < l \le k} (h_l-h_i)$, we find that 
$$
S(\mathcal{A}) \ll_{h_1,...,h_k} x\prod_{i=1}^k \frac{q_i^{k-1}}{\varphi(q_i)^k \log P^+(q_i)  }.
$$
Theorem \ref{thm4} now follows from Lemma \ref{lem:qsum} with $\alpha = k$. 

\begin{lemma}\label{lem:qsum}
Let $\alpha \in \mathbb{R}$. Assume \eqref{thetacond} and $ \theta(n) \ll n\, l(n)$ for $n\ge 1$.
We have
$$ \sum_{q\ge x, \ q \in \mathcal{B}_\theta}
\frac{q^{\alpha-1}}{\varphi(q)^\alpha \log  P^+(q) }
\ll_\alpha \frac{1}{\log x}.
$$
\end{lemma}

\begin{proof}
It suffices to estimate the sum restricted to $q\in I:=[x,x^{4/3}]$.
We write $q=mr$, where $r=P^+(q)$.
Note that $q\in \mathcal{B}_\theta \cap I$ and $\theta(n)<n^{1+o(1)}$ implies that $r\le x^{3/4}$. 
We have
$$ \sum_{q \in \mathcal{B}_\theta \cap I} \frac{q^{\alpha-1}}{\varphi(q)^\alpha \log  P^+(q) }
\ll \sum_{r\le x^{3/4}} \frac{1}{r \log r} \sum_{m  \in \mathcal{B}_\theta \cap (I/r) \atop P^+(m)\le r} 
\left(\frac{m}{\varphi(m)}\right)^\alpha \frac{1}{m}.
$$

Since $m/ \varphi(m) \ll \sigma(m)/m $, partial summation and Corollary \ref{cor1} applied to the inner sum
shows that the last expression is
$$
\ll_\alpha \sum_{r\le x^{3/4}} \frac{1}{r \log r}\cdot \frac{\log r}{\log x} \exp\left(-\frac{\log x}{3 \log r}\right)
\ll \frac{1}{\log x},
$$
by the prime number theorem.
\end{proof}

\section{The lower bound in Theorem \ref{thm3}}

\begin{lemma}\label{fLn}
Assume \eqref{thetacond} and $ \theta(n) \ll n\, l(n)$ for $n\ge 1$ . 
For $L\ge 1$ and $x\ge 1$, we have
$$
\sum_{n\in \mathcal{B}_\theta \atop n\le x} \sum_{p|n \atop p>L} \frac{\log p}{p} \ll \frac{x \log(2L)}{L\log(2x)}.
$$
\end{lemma}

\begin{proof}
As in the proof of Corollary \ref{cor1},
\begin{equation*}
\begin{split}
\sum_{n\in \mathcal{B}_\theta \atop n\le x} \sum_{p|n \atop p>L} \frac{\log p}{p} 
&=\sum_{L<p<x^{2/3}} \frac{\log p}{p} \sum_{mp \in \mathcal{B}_\theta \atop m\le x/p} 1 
 \le \sum_{L<p<x^{2/3}} \frac{\log p}{p} \sum_{m \in \mathcal{B}_{\theta_p} \atop m\le x/p} 1 \\
& \ll \sum_{L<p<x^{2/3}} \frac{\log p}{p}\cdot \frac{x \log p}{p \log(2x)}\ll \frac{x\log(2L)}{L\log(2x)},
\end{split}
\end{equation*}
by Proposition \ref{prop1} and the prime number theorem.
\end{proof}

Say a pair $n_1,n_2\in{\mathcal B}_\theta$ is $h$-$\varepsilon$-{\it special} if $\gcd(n_1,n_2)=h$ and 
$\Omega_3(n_i)\le (e+\varepsilon)\log_2 n_i$ for $i=1,2$.

\begin{lemma}\label{lempairs}
Assume \eqref{thetacond} and $n\le  \theta(n) \ll n\, l(n)$ for $n\ge 1$. 
For $h\ge 1$ satisfying \eqref{hcond} and $0<\varepsilon<1$, the number of $h$-$\varepsilon$-special pairs $n_1,n_2\in{\mathcal B}_\theta$
with $N/3<n_1,n_2<N$ and $v_2(n_1),v_2(n_2)\le C$,
where $C$ is some number depending only on $h$, is
$
\gg_{h,\varepsilon} 
{N^2 }/{\log^2 N}.
$
\end{lemma}

\begin{proof}
Write $h=2^a 3^b h'$, where $P^-(h')>3$, $a,b\ge 0$, but assume that
$a\ge 1$ or $a \ge 2$, according to 
the two cases in \eqref{hcond}. 
We consider $n_1\in \mathcal{B}_\theta$ of the form
$$ 
n_1 = 2^{a+k}3^b h' n_1' = 2^k h n_1'
$$
where $P^-(n_1')>\max\{3,P^+(h)\}=:p$ and $2^k>2p$. 
Since $\theta(n)\ge n$, the number of such $n_1$ with $N/2<n_1\le N$ is
at least
\begin{equation}\label{n1count}
 \left|\left\{ \frac{N}{h 2^{k+1}}< n_1' \le \frac{N}{h 2^k}: n_1' \in \mathcal{D}_{h2^k}, P^-(n_1')>p\right\}\right| 
\asymp_h \frac{N }{ \log N },
\end{equation}
by Lemma \ref{Saias}, for 
a suitable $k$ with $2^k>2p>2^{k+O(1)}$.  In particular, $v_2(n_1)\ll_h1$.

As in the proof of the lower bound of Theorem \ref{thm},
we can remove those $n_1$ with $\Omega_3(n_1)>(e+\varepsilon)\log_2 n_1$ without affecting
\eqref{n1count}. 
This follows from an estimate analogous to \eqref{eq:exer2}:
 $$
\sum_{n\le x}z^{\Omega_3(n)}\ll\frac{x}{3-z}\log^{z-1}x
 $$
 uniformly for $1<z<3$ (cf.\ \cite[Exercise 217(b)]{Ten}).

Let $\eta>0$ be an arbitrary constant. Lemma \ref{fLn} shows that we can choose a sufficiently large constant
$L=L(\eta)$ such that removing those $n_1$ for which 
$$\sum_{p|n_1 \atop p>L} \frac{\log p}{p} >\eta$$
will not affect \eqref{n1count}.
For each of the $\asymp_{h,\varepsilon} N/\log N$ values of $n_1$ that remain, 
consider $n_2 \in \mathcal{B}_\theta$ of the form
$$ n_2 = 2^{a}3^{b+j} h' n_2' = 3^j h n_2',$$
where $\gcd(n_2',2n_1')=1$, and $j$ is the smallest integer with $3^j>p$. 
Given $n_1$, the number of such $n_2\le N$ is at least
$$
\sum_{\substack{
 {N}/{h 3^{j+1}}<n_2' \le {N}/{h 3^j}\\ n_2' \in \mathcal{D}_{h3^j}\\\gcd(n_2',2n_1')=1}}1
\gg_h \frac{N}{ \log(N L) \log (2L)}\gg \frac{N}{ \log N} ,
$$
by Lemma \ref{lemrelp} with $p_0=3$.
As with $n_1$, this estimate is unchanged if we remove those $n_2$ with $\Omega(n_2)>(e+\varepsilon)\log_2 n_2$.   Further, $v_2(n_2)=v_2(h)\ll_h 1$.
\end{proof}

Let $N=\sqrt{xh}$.
Suppose $a, a' \in \mathcal{B}_\theta \cap (N/3,N]$ is an $h$-$\varepsilon$-special pair, with $v_2(a),v_2(a')\le C$, where $C=C(h)$ is as in Lemma \ref{lempairs}.
For each such pair $\{a,a'\}$, there is a unique pair $\{b, b'\}$ such 
that $ab-a'b'=h$ and $1\le b\le a'/h$, $1\le b' \le a/h$. We have $ab, a'b' \le aa' /h \le x$. 
Now $b,b' \le \sqrt{x/h} < 3a/h, 3a'/h$, so $ab, a'b' \in \mathcal{B}_\theta$
by the assumption on $\theta$. 
By Lemma \ref{lempairs}, it would seem we have created 
$\gg_{h,\varepsilon} x/\log^2 x$ pairs $\{ab, a'b'\}\subset \mathcal{B}_\theta\cap[1,x]$ with $ab-a'b'=h$, but we have to check for possible multiple representations.

  Note that in a graph of average degree $\ge d$, there is an induced subgraph of
minimum degree $\ge d/2$.  This folklore result can be proved by induction
on $d$, see \cite{dev}.   (Also see \cite[Prop.\ 3]{LW} for a somewhat sharper
version.)  We apply this to the graph on members of
${\mathcal B}_\theta\cap(N/3,N]$, where two integers are connected by an
edge if they form an $h$-$\varepsilon$-special pair.  From Lemma \ref{lempairs}
the average degree in this graph is $\gg N/\log N$, so there is a subgraph
$G$ of minimum degree $\gg N/\log N$.

We use this to say something about $\Omega_3(b), \Omega_3(b')$.
For edges $(a,a')$ in $G$, note that for any residue class mod~$a'$ there
are at most 2 choices for $a$, and similarly for any residue class mod~$a$
there are at most 2 choices for $a'$.
For $(a,a')$ with corresponding pair $(b,b')$ as above, let $f(a,a')=b$
and $g(a,a')=b'$.  For each fixed $a'$ the function $f$ is at most two-to-one in
the variable $a$, since
$(a/h)b\equiv 1\pmod{a'/h}$ and $b\le a'/h$.  Similarly, for each fixed $a$,
the function $g(a,a')=b'$ is at most two-to-one in the variable $a'$.  Thus, for each
fixed $a'$ there are $\gg N/\log N$ distinct values of $b$ and for each
fixed $a$ there are $\gg N/\log N$ distinct values of $b'$.  Now $b,b'\le N$
and as we have seen, the number of integers $n\le N$ with 
$\Omega_3(n)>(e+\varepsilon)\log_2 x$ is $o(N/\log N)$.  So, by possibly 
discarding $o(x/\log^2x)$ pairs $(a,a')$, we may assume that the corresponding
pair $(b,b')$ satisfies $\Omega_3(b),\Omega_3(b')\le (e+\varepsilon)\log_2 x$.

The numbers $ab$ and $a'b'$ might arise from many different pairs $(a,a')$.
However, we have $\Omega_3(ab),\Omega_3(a'b')\le2(e+\varepsilon)\log_2 x$,
so the number of odd divisor pairs of $ab,a'b'$ is 
$$
\le 2^{4(e+\varepsilon)\log_2 x}=(\log x)^{4(e+\varepsilon)\log 2}.
$$
Since $v_2(a),v_2(a')\ll_h1$,
there are $\gg_{h,\varepsilon} x/(\log x)^{2+4(e+\varepsilon)\log 2}$ pairs $n,n+h\in{\mathcal B}_\theta$ with $n\le x$.  This completes the proof of the theorem.


%

\bigskip

\noindent{\bf Acknowledgments}.
We thank David Eppstein for informing us of \cite{LW} and
Paul Pollack for \cite{Sun}.

\end{document}